\documentclass[11pt]{amsart}
\usepackage{amsmath,tikz}
\usepackage{amsfonts}

\usepackage{amscd}
\usepackage{amsthm}
\usepackage{amssymb} \usepackage{latexsym}
\usepackage{eufrak}
\usepackage{euscript}
\usepackage{epsfig}
\usepackage{graphics}
\usepackage{array}
\usepackage{enumerate}
\usepackage{dsfont}
\usepackage{color}
\usepackage[colorlinks,citecolor=red,pagebackref,hypertexnames=false, breaklinks]{hyperref}

\newcommand{\R}{{\mathbb R}}

\newcommand{\cha}{\mathds{1}}

\newtheorem{lemma}{Lemma}[section]
\newtheorem{definition}[lemma]{Definition}

\newtheorem{theorem}[lemma]{Theorem}
\newtheorem{corollary}[lemma]{Corollary}
\newtheorem{example}{Example}

\newtheorem{thm}{Theorem}

\newtheorem{problem}{Problem}

\numberwithin{equation}{section}

\DeclareFixedFont{\Acknowledgment}{OT1}{cmr}{bx}{n}{14pt}
\begin{document}

\title{\bf A note on limit of first eigenfunctions of $p$-Laplacian on graphs}
\date{}

\author{Huabin Ge}
\address{Huabin Ge, Department of Mathematics, Beijing Jiaotong University, Beijing, 100044, China}
\email{hbge@bjtu.edu.cn}
\thanks{H. Ge is supported by NSFC (China) under grant no. 11871094.}

\author{Bobo Hua}
\address{Bobo Hua, School of Mathematical Sciences, LMNS, Fudan University, Shanghai 200433, China}
\email{bobohua@fudan.edu.cn}
\thanks{B. Hua is supported by NSFC (China) under grant no. 11831004 and grant no. 11826031. }

\author{Wenfeng Jiang}
\address{Wenfeng Jiang, School of Mathematics (Zhuhai), Sun Yat-Sen University, Zhuhai, China}
\email{wen\_feng1912@outlook.com}
\thanks{W. Jiang is supported by Chinese Universities Scientific Fund(74120-31610002)}

\maketitle

\begin{abstract}
We study the limit of first eigenfunctions of (discrete) $p$-Laplacian on a finite subset of a graph with Dirichlet boundary condition, as $p\to 1.$ We prove that up to a subsequence,  they converge to a summation of characteristic functions of Cheeger cuts of the graph. We give an example to show that the limit may not be a characteristic function of a single Cheeger cut.
\end{abstract}

\setcounter{section}{-1}
\section{Introduction}

The spectral theory of linear Laplacian, called Laplace-Beltrami operator, on a domain of an Euclidean space or a Riemannian manifold is extensively studied in the literature, see e.g. \cite{CH53}. The $p$-Laplacians are nonlinear generalizations of the linear Laplacian, which corresponds to the case $p=2$. These are nonlinear elliptic operators which possess many analogous properties as the linear Laplacian.

A graph consists of a set of vertices and a set of edges. The Laplacian on a finite graph is a finite dimensional linear operator, see e.g. \cite{Chung10}, which emerges from the discretization of the Laplace-Beltrami operator of a manifold, the Cayley graph of a discrete group, data sciences and many others. Compared to continuous Laplacians, one advantage of discrete Laplacians is that one can calculate the eigenvalues of the Laplacian on a finite graph which are in fact eigenvalues of a finite matrix.


Cheeger \cite{Cheeger70} defined an isoperimetric constant, now called Cheeger constant, on a compact manifold and used it to estimate the first non-trivial eigenvalue of the Laplacian, see also \cite{K}. These were generalized to graphs by Alon-Milman \cite{Alon-Milman} and Dodziuk \cite{Dodziuk84} respectively. These estimates, called Cheeger estimates, are very useful in the spectral theory. The spectral theory for discrete $p$-Laplacians was studied by \cite{Yamasaki79,Amghibech03,Takeuchi03,Hein-Buhler10,Keller-Mugnolo16}. It turns out that this theory unifies these constants involved in the Cheeger estimate. The Cheeger constant of a finite graph is in fact the first non-trivial eigenvalue of $1$-Laplacian, see Chang, Hein et al. \cite{Chang16,Chang-Shao-Zhang17,Hein-Buhler10}. So that the Cheeger estimate can be regarded as the eigenvalue relation for $p$-Laplacians, $p=1$ and $p=2.$

In this paper, we study first eigenvalues and eigenfunctions of the $p$-Laplacian on a subgraph of a graph with Dirichlet boundary condition. Let $G= (V, E)$ be a simple, undirected, locally finite graph, where $V$ is the vertex set and $E$ is the edge set. Two vertices $x,y$ are called neighbors, denoted by $x\sim y$, if $\{x,y\}\in E.$ Let $$\mu:V \to (0, \infty), x\mapsto \mu_x,$$ be the vertex measure on $V$ and $$w: E \to (0, \infty), \{x,y\}\mapsto w_{xy}=w_{yx},$$ be the edge measure on $E.$ The quadruple $(V,E,\mu,w)$ is called a weighted graph. For a subset of $V$ (a subset of $E,$ resp.) we denote by $|\cdot|_\mu$ ($|\cdot|_w,$ resp.) the $\mu$-measure ($w$-measure, resp.) of the set. Let $\Omega$ be a finite subset of $V.$ We denote by $C_0(\Omega)$ the set of function on $V$ which vanishes on $V\setminus \Omega.$ 

For $p>1,$ we define the $p$-Laplacian with Dirichlet boundary condition, called Dirichlet $p$-Laplacian, on $\Omega$ as 
$$\Delta_p u(x)=\frac{1}{\mu_x}\sum_{y\in V: y\sim x}w_{xy}|u(y)-u(x)|^{p-2}(u(y)-u(x)),\quad x\in \Omega, u\in C_0(\Omega).$$ We call the pair $(\lambda, u)\in \R\times C_0(\Omega)$ satisfying 
\begin{equation}\label{defi:eigenequation}-\Delta_p u(x)=\lambda |u(x)|^{p-2}u(x),\quad \forall x\in \Omega\end{equation} the eigenvalue and the eigenfunction of Dirichlet $p$-Laplaican.
The smallest eigenvalue of Dirichlet $p$-Laplacian is called the first eigenvalue, denoted by $\lambda_{1,p}(\Omega)$, and the associated eigenfunction is called the first eigenfunction, denoted by $u_p.$

For $p\geq 1,$ the $p$-Dirichlet energy is defined as
$$E_p(u):=\sum_{x,y\in V, y\sim x}w_{xy}|u(y)-u(x)|^p,\quad u\in C_0(\Omega),$$ and the modified $p$-Dirichlet energy is defined as $$\widetilde{E_p}(u):=\frac{E_p(u)}{|u|^p},\quad \forall u\in C_0(\Omega)\setminus \{0\}.$$
We consider the variational problem \begin{equation}\label{defil:variation}\widetilde{E_p}:C_0(\Omega)\setminus\{0\}\to \R.\end{equation}
One is ready to see that for $p>1,$ critical points and critical values of the problem satisfy \eqref{defi:eigenequation}. The spectral theory for $1$-Laplacian was developed by Chang and Hein. Eigenvalues and eigenfunctions of Dirichlet $1$-Laplacian are defined as critical values and critical points of the variational problem \eqref{defil:variation} for $p=1.$ We denote by $\lambda_{1,1}(\Omega)$ the first eigenvalue for Dirichlet $1$-Laplacian on $\Omega.$ 
This yields the Rayleigh quotient characterization, $p\geq 1,$
\begin{equation}\label{eq:Rayleigh}\lambda_{1,p}(\Omega)=\inf_{u\in C_0(V), u\neq 0}\widetilde{E_p}(u).\end{equation}  First eigenfunctions of Dirichlet $p$-Laplacian have some nice properties, see \cite{BH}.
\begin{thm} For a finite connected subset $\Omega$ of $V$ and $p>1,$ a first eigenfunction $u$ on $\Omega$ has fixed sign, i.e. $u>0$ on $\Omega$ or $u<0$ on $\Omega.$ Moreover, the first eigenvalue is simple, i.e. for any two first eigenfunctions $u$ and $v$ on $\Omega,$ $u=c v$ for some $c\neq 0.$
\end{thm} By this result, there is a unique first eigenfunction on $\Omega$ satisfies
$$u>0,\quad \mathrm{on}\ \Omega,\quad  |u|_p=1,$$ which is called the normalized first eigenfunction of $p$-Laplacian.

We introduce the definition of the Cheeger constant on $\Omega.$ For any subset $D\in V,$ the boundary of $D$ is defined as $\partial D:=\{\{x,y\}\in E: x\in D, y\in V\setminus D\}.$
 \begin{definition}
 The Cheeger constant on $\Omega$ is defined as
 $$
 h(\Omega)=\inf_{D\subset \Omega} \frac{|\partial D|_w}{|D|_\mu}.
 $$ A subset $D$ of $\Omega$ is called a Cheeger cut if
 $$
 \frac{|\partial D|_w}{|D|_\mu}=h(\Omega).
 $$
 \end{definition}

We prove the main result in the following.
\begin{theorem}\label{thm:main1} Let $\Omega$ be a finite connected subset of $V.$
For any sequence $\{p_i\}_{i=1}^\infty$ satisfying $p_i>1, p_i\to 1$, let $u_i$ be the corresponding normalized first eigenfunction of $p_i$-Laplacian. Then there is a subsequence $\{p_{i_k}\}$ such that $\{u_{i_k}\}_k$ converges
$$
\lim_{k \to \infty}u_{i_k}=\sum_{n=1}^N c_n\cha_{A_n} .
$$
where $A_n,$ $1\leq n\leq N$, are Cheeger cuts of $\Omega$ satisfying $A_{N}\subsetneq A_{N-1}\subsetneq \cdots \subsetneq A_1, $ $\cha_{A_n}$ are characteristic functions on $A_n$ and $c_n>0.$  \end{theorem}

We have the following corollary.
\begin{corollary}\label{coro:1}For a finite connected subset $\Omega$ of $V,$ suppose that the Cheeger cut of $\Omega$ is unique. Then 
$$\lim_{p \to \infty}u_{p}=\frac{1}{|A|_\mu} \cha_A,$$ where $A$ is the Cheeger cut of $\Omega.$ 
\end{corollary}

Concerning with these results, we have the following open problems.
\begin{problem}\label{prob:1} In general case, is it true that the normalized first eigenfunction $u_p$ converges, as $p\to 1$?
\end{problem}
\begin{problem}\label{prob:2} What are the limits of the normalized first eigenfunctions in Theorem~\ref{thm:main1}.
\end{problem} 

For Problem~\ref{prob:1}, since there is no uniqueness for the Cheeger cuts, see e.g. Example~\ref{exam:1} in Section~\ref{sec:example}, one needs new ideas to prove the result.
For Problem~\ref{prob:2}, one might hope that the limit of a sequence of normalized first eigenfunctions is a characterization function of a single Cheeger cut, as in Corollary~\ref{coro:1}. By investigating Example~\ref{exam:1} in Section~\ref{sec:example}, we show that this is not true in general. This indicates that the result in Theorem~\ref{thm:main1} cannot be improved to the characterization function of a single Cheeger cut.

The paper is organized as follows: In next section, we prove the main result, Theorem~\ref{thm:main1}. In Section~\ref{sec:example}, we construct an example to show the sharpness of Theorem~\ref{thm:main1}.

\section{Proof of Theorem~\ref{thm:main1}}
Let $(V,E,\mu,w)$ be a weighted graph and $\Omega$ is a connected subset of $V,$ i.e. for any $x,y\in \Omega$ there is a path, $x=x_0\sim x_1\sim \cdots \sim x_k=y$, connecting $x$ and $y$ with $x_i\in \Omega,\forall 1\leq i\leq k-1.$ We need some lemmas.

\begin{lemma} $$\lim_{p\to 1}\lambda_{1,p}(\Omega)=\lambda_{1,1}(\Omega).$$
\end{lemma}
\begin{proof} Note that for any $u\in C_0(\Omega),$ $u\neq 0,$ 
$$\widetilde{E_p}(u)\to \widetilde{E_1}(u),\quad p\to 1.$$ Hence the lemma follows from the Rayleigh quotient characterization \eqref{eq:Rayleigh}.
\end{proof}
\begin{lemma}\label{lem:1.2} For any $p\geq 1,$
$$
\lambda_{1,p}(\Omega)\leq h(\Omega).
$$
\end{lemma}
\begin{proof}
As $\Omega$ is a finite subset, we choose $D\subset \Omega$ such that
$
\frac{|\partial D|_w}{|D|_\mu}=h(\Omega).
$
Consider the characteristic function on $D,$ $$\cha_D(x)=\left\{\begin{array}{ll}1,& x\in D,\\
0,& x\in V\setminus D,\end{array}\right.$$
Then by \eqref{eq:Rayleigh},
$$\lambda_{1,p}(\Omega)\leq \widetilde{E_p}(\cha_D)=\frac{|\partial D|_w}{|D|_\mu}=h(\Omega).$$ This proves the lemma.

\end{proof}

\begin{lemma}\label{lem:equivalent}
 
$\lambda_{1,1}(\Omega)= h(\Omega).$
\end{lemma}
\begin{proof} By Lemma~\ref{lem:1.2}, it suffices to prove that $\lambda_{1,1}(\Omega)\geq h(\Omega).$
 For any $u\in C_0(\Omega)$ with $|u|_1=1,$ let $g=|u|.$ For any $\sigma\geq 0,$ set $\Omega_{\sigma}:=\{x\in \Omega: g(x)>\sigma\}$ and $G(\sigma):=|\partial \Omega_\sigma|_w.$ Then
$$
G(\sigma)=\sum_{e=\{x,y\}\in E,g(x)\leq \sigma< g(y)}w_{xy}.
$$

Hence \begin{eqnarray*}
\int_{0}^{+\infty}G(\sigma) d\sigma&=&\int_{0}^{+\infty}\sum_{e=\{x,y\}\in E,g(y)>g(x)}w_{xy} \cha_{[g(x),g(y))}(\sigma)d\sigma\notag\\
&=&\sum_{e=\{x,y\}\in E,g(y)>g(x)}w_{xy}\int_{0}^{+\infty} \cha_{[g(x),g(y))}(\sigma)d\sigma\notag\\
&=& \sum_{e=\{x,y\}\in E,g(y)>g(x)}w_{xy}|g(x)-g(y)|=E_1(g). \notag\\
\end{eqnarray*}
Moreover, by the definition of $h(\Omega)$, $\Omega_{\sigma}\leq h(\Omega) | \Omega_{\sigma}|_\mu,\ \forall \sigma\geq 0.$ This yields that
\begin{align}\label{main-proof-2}
E_1(u)\geq E_1(g)=&\int_{0}^{+\infty}G(\sigma)d\sigma\geq h(\Omega) \int_{0}^{+\infty}| \Omega_{\sigma}|_\mu\\
=&h(\Omega)|u|_1=h(\Omega).\notag
\end{align} By taking the infimum over $u\in C_0(\Omega)$ with $|u|_1=1$ on the left hand side, we prove the lemma by \eqref{eq:Rayleigh}.

\end{proof}

By the definition, $u$ is called a first eigenfunction of $1$-Laplacian on $\Omega$ if $$\widetilde{E_1}(u)=\inf_{v\in C_0(V), v\neq 0}\widetilde{E_1}(v)=\lambda_{1,1}(\Omega).$$ Since $\widetilde{E_1}(|u|)\leq \widetilde{E_1}(u),$ $|u|$ is also a first eigenfunction of $1$-Laplacian on $\Omega.$

\begin{lemma}\label{lem:structure} Let $u$ be a nonnegative first eigenfunction of $1$-Laplacian on $\Omega.$ Then
for any $\sigma\geq 0,$ $\Omega_\sigma:=\{x\in \Omega: u>\sigma\}$ is a Cheeger cut. Moreover, \begin{equation}\label{eq:structure1}
u=\sum_{n=1}^N c_n\cha_{A_n} .
\end{equation}
where $A_n,$ $1\leq n\leq N$, are Cheeger cuts of $\Omega$ satisfying $A_{N}\subsetneq A_{N-1}\subsetneq \cdots \subsetneq A_1,$ and $c_n>0.$
\end{lemma}
\begin{proof} Let $\{a_i\}_{i=0}^M$ be the range of the function $u,$ such that
$$0=a_0<a_1<\cdots<a_M=\max_{x\in V}u(x).$$ Then for any $\sigma\in [a_i, a_{i+1}),$ $0\leq i\leq M-1,$ $$\Omega_\sigma=\Omega_{a_i}.$$ By the proof of Lemma~\ref{lem:equivalent}, see \eqref{main-proof-2},
\begin{eqnarray*}h(\Omega)|u|_1&=&\lambda_{1,1}(\Omega)|u|_1=E_1(u)=\int_{0}^{+\infty}G(\sigma)d\sigma=\sum_{i=0}^{M-1}(a_{i+1}-a_i)|\partial \Omega_{a_i}|_w\\
&\geq& h(\Omega) \sum_{i=0}^{M-1}(a_{i+1}-a_i)|\Omega_{a_i}|_\mu=h(\Omega)|u|_1.\end{eqnarray*} Hence for any $0\leq i\leq M-1,$ $$|\partial \Omega_{a_i}|_w=h(\Omega)|\Omega_{a_i}|_\mu.$$ So that for any $\sigma\geq 0,$
$\Omega_\sigma$ is a Cheeger cut. 

For the other assertion, note that 
$$u=a_1\cha_{\Omega_{a_0}}+(a_2-a_1)\cha_{\Omega_{a_1}}+\cdots+(a_M-a_{M-1})\cha_{\Omega_{M-1}}.$$ This proves the result.

\end{proof}

Now we are ready to prove Theorem~\ref{thm:main1}.
\begin{proof}[Proof of Theorem~\ref{thm:main1}] Note that $$|u_i|_{p_i}=1,\quad \widetilde{E_1}(u_i)=\lambda_{1,p_i}(\Omega).$$
For any $x\in \Omega,$ $|u_i(x)|\leq |u_i|_{p_i}=1.$ Since $\Omega$ is a finite set, there is a subsequence $\{p_{i_k}\}_k$ of $\{p_i\}_i$ and $u\in C_0(\Omega)$ such that
$$u_{i_k}\to u,\quad \mathrm{pointwise\ on\  \Omega}.$$ Hence
$$u\geq 0,\quad |u|_1=\lim_{k\to \infty}|u_{i_k}|_{p_{i_k}}=1,$$$$\quad E_1(u)=\lim_{k\to \infty}E_1(u_{i_k})=\lim_{k\to \infty}\lambda_{1,p_{i_k}}(\Omega)=\lambda_{1,1}(\Omega).$$
This yields that $u$ is a nonnegative normalized $1$-Laplacian eigenfunction. The theorem follow from Lemma~\ref{lem:structure}.

\end{proof}

This yields the following corollary.
\begin{proof}[Proof of Corollary~\ref{coro:1}] It suffices to prove that for any sequence $p_i\to 0,$ $p_i>1,$ there is a subsequence $p_{i_k}$ such that $$\lim_{k \to \infty}u_{p_{i_k}}=\frac{1}{|A|_\mu} \cha_A,$$ where $A$ is the Cheeger cut of $\Omega.$ By Lemma~\ref{lem:structure}, there is a subsequence $p_{i_k}$ and a nonnegative normalized first eigenfunction of $1$-Laplacian $u$ such that
$$\lim_{k \to \infty}u_{p_{i_k}}=u.$$ Since the Cheeger cut of $\Omega$ is unique, by \eqref{eq:structure1}, $u=c\cha_A.$ By $|u|_1=1,$ $c=\frac{1}{|A|_\mu}.$ This proves the corollary.
\end{proof}

\section{An example}\label{sec:example}
In this section, we give an example to show that the limit of the first eigenfunction of $\Delta_p$ may not be the characteristic function of a single set.

\begin{example}\label{exam:1}Consider the graph $G$ as in Fig.~\ref{Fig:1} such that edge weight $w_{x,y}=1$ for any $x\sim y$ and vertex weight $\mu_{x_1}=\mu_{x_2}=2,\mu_{y_1}=\mu_{y_2}=4.$ Let $\Omega=\{x_1,x_2,y_1,y_2\}.$ By the enumeration, the Cheeger cuts for $h(\Omega)$ are $$\{x_1,x_2\},\{x_1,x_2,y_1\},\{x_1,x_2,y_2\},\{x_1,x_2,y_1,y_2\}.$$ 
\end{example}
\begin{figure}[htbp]
 \begin{center}
   \begin{tikzpicture}
    \node at (0,0){\includegraphics[width=0.84\linewidth]{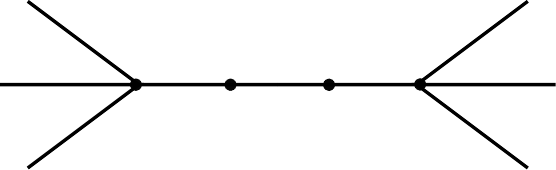}};
    \node at (-3.2,   -.5){\Large $y_1$};
        \node at (-1,   -.5){\Large $x_1$};
                \node at (1.15,   -.5){\Large $x_2$};
    \node at (3.2,   -.5){\Large $y_2$};
   \end{tikzpicture}
  \caption{Example $G$}\label{Fig:1}
 \end{center}
\end{figure}
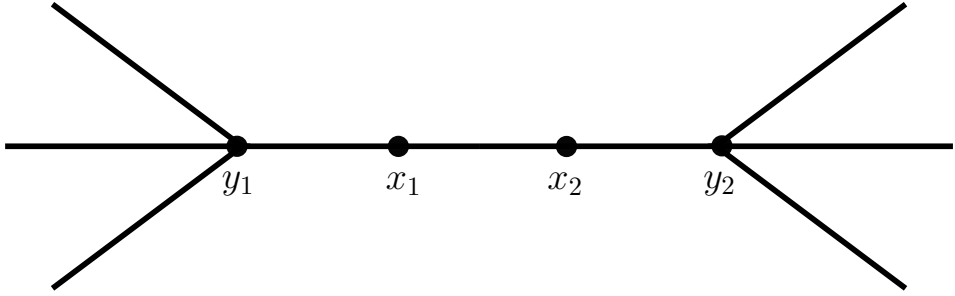 

We calculate the normalized first eigenfunctions $u_p$ of $p$-Laplacian on $\Omega$ for $p>1.$ One is ready to see that there is a symmetry, $T:\Omega\to \Omega,$ such that $$T(x_1)=x_2,T(x_2)=x_1,T(y_1)=y_2,T(y_2)=y_1,$$ perserving the Dirichlet boundary condition. Since the first eigenfunction is positive and unique up to constant multiplication,  $u_p(x_1)=u_p(x_2),u_p(y_1)=u_p(y_2),$ see \cite{BH}. For convenience, we write $u_p$ as a vector $$u_p=(u_p(x_1),u_p(x_2),u_p(y_1),u_p(y_2)),$$ and by scaling we set
$$v_p=\frac{1}{u_p(x_1)}u_p=(1,1,t_p,t_p).$$ Note that $v_p$ is also a first eigenfunction of  $p$-Laplacian on $\Omega.$

Setting $x=t_p,$ by the eigen-equation, 
\begin{equation}\left\{\begin{array}{ll}-\frac{1}{2}|1-x|^{p-2}(x-1)=\lambda_{1,p},\\
-\frac14(3|x|^{p-2}(-x)+\frac{1}{2}|1-x|^{p-2}(1-x))=\lambda_{1,p} x^{p-1}.
\end{array}\right.
\end{equation} By the first equation, $x<1.$ Plugging the first equation into the second and dividing it by $x^{p-1}$, we get
$$2(1-x)^q+(\frac1x-1)^q-3=0,$$ where $q=p-1.$

Set $$f(x,q)=2(1-x)^q+(\frac1x-1)^q-3,\quad \forall x\in (0,1], q>0.$$ For any $q>0,$ since $f(\cdot, q)$ is monotonely decreasing, $f(1,q)=-3$ and $f(x,q)\to +\infty, x\to 0+$, there is a unique solution to $f(x,q)=0$ denoted by $x_q.$ Note that $x_q=t_p.$

For fixed $x\in (0,1),$ we write $a=1-x,$ $b=\frac1x-1,$ 
and set $$g(q)=f(x,q)=2 a^q+b^q-3.$$
For $q\geq 0,$ $$g'(q)=2 a^q\ln a+b^q\ln b ,\quad g''(q)=2 a^q(\ln a)^2+ b^q(\ln b)^2\geq 0,$$ where the derivatives are understood as right derivatives for $q=0.$
Hence $g(p)$ is convex in $(0,+\infty),$ which yields that for $q\geq 0,$ \begin{equation}\label{eq:exam1}g(q)\geq g(0)+g'(0)q= \ln(a^2 b) q.\end{equation}
One can show that 

\begin{equation*}\left\{\begin{array}{ll}a^2 b>1,& x<\hat{x},\\
a^2 b=1,& x=\hat{x},\\
a^2 b<1,& x>\hat{x},\\
\end{array}\right.
\end{equation*} where $\hat{x}$ is the real solution of $(1-x)^3=x,$ given by $$\hat{x}=1-\sqrt[3]{\frac{\sqrt{93}+9}{18}}+\sqrt[3]{\frac{\sqrt{93}-9}{18}}\approx 0.31767.$$ Hence by \eqref{eq:exam1}, for any $x<\hat{x},$
$$g(q)\geq \ln(a^2 b) q>0,\quad q>0.$$ Hence $x_q\geq \hat{x}.$ 

For any $x>\hat{x},$ by Taylor expension of $g(q)$ at $q=0,$ 
$$g(q)= \ln(a^2 b) q+o(q),\quad q\to 0.$$ Hence by $a^2b<1,$ for sufficiently small $q,$
$$g(q)<0,$$ this yields that $x_q<x.$ This implies that
$\limsup_{q\to 0} x_q\leq x.$ By passing to the limit $x\to \hat{x}+,$ we get $$\limsup_{q\to 0} x_q\leq \hat{x}.$$ Combining it with $x_q\geq \hat{x},$ $$\lim_{q\to 0} x_q=\hat{x}.$$  

This yields that as $p\to 1,$
$$v_p=(1,1, t_p,t_p)\to (1,1, \hat{x},\hat{x}).$$ This yields that
$$u_p=\frac{v_p}{|v_p|_p}\to \frac{1}{4+8\hat{x}}(1,1, \hat{x},\hat{x})\approx (0.15287,0.15287,0.04856,0.04856).$$
Hence, the limit is not a characteristic function on a single set.


\begin{thebibliography}{50}

\setlength{\itemsep}{-2pt} \small


\bibitem{Alon-Milman}
\textit{N. Alon and V. D. Milman},
$\lambda_1$, isoperimetric inequalities for graphs, and superconcentrators.
J. Combin. Theory Ser. B, \textbf{38}(1985), 73--88.

\bibitem{Amghibech03}
\textit{S. Amghibech},
Eigenvalues of the discrete p-Laplacian for graphs.
Ars Combin., \textbf{67}(2003), 283--302.

\bibitem{Chang16}
\textit{K. C. Chang},
Spectrum of the 1-Laplacian and Cheeger's constant on graphs,
J. Graph Theory, \textbf{81}(2016), no.2, 167--207.

\bibitem{Chang-Shao-Zhang17}
\textit{K. C. Chang, S. H. Shao, D. Zhang},
Cheeger's cut, maxcut and the spectral theory of $1$-Laplacian on graphs,
Sci. China Math. (2017), no. 11, 1--18.

\bibitem{Cheeger70}
\textit{J. Cheeger},
A lower bound for the smallest eigenvalue of the Laplacian,
Problems in analysis (Papers dedicated to Salomon Bochner, 1969),
Princeton Univ. Press, Princeton, N. J., 1970, 195--199.

\bibitem{Chung10}
\textit{F. Chung},
Spectral Graph Theory,
Betascript Publishing,
2010:212.










\bibitem{CH53}
\textit{R. Courant and D. Hilbert},
Methods of mathematical physics. Vol. I. Interscience Publishers, Inc., New York, N.Y., 1953.







\bibitem{Dodziuk84}
\textit{J. Dodziuk},
Difference equations, isoperimetric inequality and transience of certain random walks,
Trans. Amer. Math. Soc., \textbf{284}(1984), 787--794.

\bibitem{Hein-Buhler10}
\textit{M. Hein, T. B\"{u}hler},
An inverse power method for nonlinear eigenproblems
with applications in 1-spectral clustering and sparse PCA.
In Advances in Neural Information Processing Systems (NIPS),
847--855, Cambridge, MA, 2010. MIT Press.

\bibitem{BH} \textit{B. Hua and L. Wang}, {Dirichlet $p$-Laplacian eigenvalues and Cheeger constants on symmetric graphs}, preprint.

 
 
\bibitem{K} \textit{B. Kawohl}, {Isoperimetric estimates for the first eigenvalue of the p-Laplace operator and the Cheeger constant.},   Commentationes Mathematicae Universitatis Carolinae (2003) Volume: 44, Issue: 4, page 659-667
ISSN: 0010-2628.

\bibitem{Keller-Mugnolo16}
\textit{M. Keller, D. Mugnolo},
General Cheeger inequalities for $p$-Laplacians on graphs,
Nonlinear Analysis, \textbf{147}(2016), 80--95.

\bibitem{Takeuchi03}
\textit{H. Takeuchi},
The spectrum of the p-Laplacian and $p$-harmonic morphisms on graphs,
Illinois J. Math., \textbf{47}(2003), 939--955.

\bibitem{Yamasaki79}
\textit{M. Yamasaki},
Discrete potentials on an infinite network,
Memoirs of the Faculty of Literature and Science, Shimane University,
\textbf{13}(1979), 31--44.

\end{thebibliography}
\end{document}